\documentclass[reqno,11pt]{amsart}
\usepackage[top=1.2in,bottom=1.0in,left=1.0in,right=1.0in]{geometry}
\usepackage{amsmath}
\usepackage{amssymb}
\usepackage{mathrsfs}
\usepackage{times}
\usepackage{bm}
\usepackage{latexsym}
\usepackage{graphicx}
\usepackage{indentfirst}
\usepackage{graphicx}
\usepackage{epsfig}
\usepackage{psfrag}
\usepackage{subfigure}
\usepackage{url}
\usepackage{caption}
\usepackage{float}
\usepackage{stfloats}
\usepackage{fancyhdr}
\usepackage{dsfont}
\usepackage{amsfonts}
\usepackage{enumerate}
\usepackage{epstopdf}
\usepackage{color}
\usepackage{amssymb,amsmath,amsthm,mathrsfs}
\usepackage{booktabs}
\usepackage{threeparttable}
\usepackage[titletoc]{appendix}
\usepackage[linesnumbered,ruled,vlined]{algorithm2e}
\usepackage{multirow}
\usepackage{array}
\usepackage[misc]{ifsym}

    \numberwithin{equation}{section}%
    \numberwithin{table}{section}%
    \numberwithin{figure}{section}
\allowdisplaybreaks

\newtheorem{lemma}{Lemma}[section]
\newtheorem{remark}{Remark}[section]

\newtheorem{proposition}{Proposition}[section]
\newtheorem{example}{Example}

\DeclareMathAlphabet{\mathmatrix}{OT1}{ptm}{b}{n}
\DeclareMathAlphabet{\mathvector}{OT1}{ptm}{bx}{it}

\def\d{\mathrm{d}}
\def\RR{\mathbb{R}}

\def\NN{\mathbb{N}}

\def \bx{\bm x}

\begin{document}

\title[]{ Preconditioned  Legendre spectral Galerkin methods for the non-separable elliptic equation }


\author{Xuhao Diao}
\address[Xuhao Diao]{ School of Mathematical Sciences, Peking University, Beijing 100871, China.}
\email{diaoxuhao@pku.edu.cn}

\author{Jun Hu}
\address[Jun Hu]{ School of Mathematical Sciences, Peking University, Beijing 100871, China.} \email{hujun@math.pku.edu.cn}

\author{Suna Ma}
\address[Suna Ma \Letter]{ School of Mathematical Sciences, Peking University, Beijing 100871, China.} \email{masuna@pku.edu.cn}



\maketitle

\begin{abstract}
The Legendre  spectral  Galerkin method of self-adjoint second order elliptic equations usually results in a  linear system with a dense and ill-conditioned coefficient matrix. In this paper, the linear system is solved by a preconditioned conjugate gradient (PCG) method where the preconditioner $M$ is  constructed by approximating the variable coefficients with a ($T$+1)-term Legendre series in each direction to a desired accuracy. A feature of the proposed PCG method is that the iteration step increases slightly with the size of the resulting matrix when reaching  a certain approximation accuracy. The efficiency of the method lies in that the system with the preconditioner $M$ is approximately solved by a one-step iterative method based on the ILU(0) factorization. The ILU(0) factorization
of $M\in \RR^{(N-1)^d\times(N-1)^d}$ can be computed using $\mathcal{O}(T^{2d} N^d)$ operations, and the number of nonzeros in the  factorization factors is of $\mathcal{O}(T^{d} N^d)$, $d=1,2,3$.
To further speed up  the PCG method,  an algorithm is developed for fast
matrix-vector multiplications by the resulting  matrix of Legendre-Galerkin spectral discretization, without the need to explicitly form it. The complexity of the fast matrix-vector multiplications is of  $\mathcal{O}(N^d (\log N)^2)$.
As a result, the PCG method has a $\mathcal{O}(N^d (\log N)^2)$ total complexity for a $d$ dimensional domain with $(N-1)^d$ unknows, $d=1,2,3$.
Numerical examples are given to demonstrate the efficiency of proposed preconditioners and  the algorithm for fast matrix-vector multiplications.
\end{abstract}

\section{introduction}

Spectral methods are an important tool in engineer and scientific computing  for solving differential equations due to their high order of accuracy; see \cite{Canuto2006,ShenTang2006,GottliebOrszag1977,ShenTangWang2011} and the references therein. However, for problems with general variable coefficients,  spectral methods lead to a linear system with a dense and ill-conditioned matrix. Moreover, the dense matrix is usually not explicitly available, since it is costly to form it.
In practice, it becomes rather prohibitive to solve the linear system by a direct solver or even an iterative method without preconditioning for the multi-dimensional cases, when the size of the matrix is large.

Over the years there has been intensive research  on
the spectral collocation method for solving problems with variable coefficients, since it is easy to implement, once the differentiation matrices are precomputed. One significant attempt is to use a lower-order method (finite differences or finite elements \cite{CanutoQuarteroni1985, DevilleMund1985,KimParter1996, KimParter1997,FangShen2018}) or integration operator \cite{Coutsias1996,Hesthaven1998,WangSamson2014} as a preconditioner
and to take advantage of the fact that the matrix-vector multiplication from a Fourier- or Chebyshev-spectral discretization can be performed in a quasi-optimal complexity.
Another  approach is the finite element  multigrid preconditioning method proposed by
Shen et al.\cite{ShenWangXu2000} for the Chebyshev-collocation approximation of  second-order elliptic equations.  Although  many spectral collocation methods have been applied to numerically solve variable-coefficient differential equations, few efforts  are found for   spectral Galerkin methods, especially Legendre-Galerkin methods,  in literature. An early work is
the Chebyshev-Legendre Galerkin method for second-order elliptic problems  introduced in \cite{shen1996}, which is based on the Legendre-Galerkin formulation, and only the coefficients of Legendre expansions and the values at the Chebyshev-Gauss-Lobatto  points are used in the computation.
 A fast direct solver was presented  for the Legendre-Galerkin approximation of the two and three dimensional Helmholtz equations by Shen in \cite{Shen1994Legendre}, whose complexity is of $\mathcal{O}(N^{d+1})$  in a $d$ dimensional domain. An improved two-dimensional algorithm  was constructed by further exploring the matrix structures of the Legendre-Galerkin discretization in \cite{Shen1995}, whose complexity is of $\mathcal{O}(N^{2}\log_2 N)$, which was extended to the Legendre-Galerkin spectral approximation of the three dimensional Helmholtz equation in \cite{Auteri2000}.

The Legendre-Galerkin method of self-adjoint second order elliptic equations leads to symmetric  linear systems, but its efficiency is limited by the lack of fast transforms between the physical space (values at the Legendre Gauss points) and the spectral space (coefficients of the Legendre polynomials). In traditional spectral methods, a fast algorithm for  Legendre expansions is a procedure to fast  evaluate the Legendre expansion at Chebyshev points, and conversely, to fast  evaluate  the coefficients of the Legendre expansion from the table of its values at the Chebyshev-Gauss-Lobatto   points. Recently, a series of work were done for fast discrete Legendre transforms between the values at the Legendre Gauss points and the coefficients of the Legendre polynomials \cite{Tygert2010, potts2003, Keiner2009,HaleTownsend2015}.
In particular,  an $\mathcal{O}(N (\log N)^2/\log\log N)$ algorithm based on the FFT was proposed  in \cite{HaleTownsend2015} in one dimension for computing the  discrete Legendre transform with a degree $N-1$ Legendre expansion  at $N$ Legendre points.

The goal of this article is to fast solve the linear system  resulting from the Legendre-Galerkin spectral discretization of second order elliptic equations with variable coefficients by the preconditioned conjugate gradient (PCG) method. The  novelties of the paper lie in the following three folds:
\begin{itemize}
\item
Firstly, the preconditioner is constructed by using a truncated Legendre series to approximate the variable coefficients.
 It is in the  case that  the iterative step of the PCG method only increase slightly with  the size of the resulting matrix.
 A closely related preconditioner of  \cite{shen1996, shen1997} is to use a constant-coefficient problem to  precondition variable-coefficient problems. However, for coefficients with large variations,  iterative methods with that preconditioner usually converge very slowly \cite{ShenWangXu2000}.
\item
Secondly, by means of fast discrete Legendre transforms,  an algorithm is developed for fast matrix-vector multiplications by the resulting matrix  without the need to explicitly form it. As a result, they can be done in $\mathcal{O}(N^d (\log N)^2)$ operations.
\item
Last but not least, the system with the preconditioner as the coefficient matrix is approximately solved by a one-step iterative method based on the ILU(0) factorization. Thanks to  the sparse structure of the preconditioner $M$, the ILU(0) factorization gives an  unit lower triangular matrix $L$ and an upper triangular matrix $U$, where together the $L$ and $U$ matrices have the same number of nonzero elements as the matrix $M$. The complexity essentially depends on the number of nonzeros in $M$, which is of $\mathcal{O}(T^{2d} N^d)$ with $T$ the cutoff number of the Legendre series in each direction  used to approximate the coefficient functions.

\end{itemize}

The remainder of this article is organized as follows.  Some preliminaries are  given in Section 2. Section 3 introduces the Legendre-Galerkin method  of   second-order  elliptic equations with  non-separable coefficients. In Section 4, the  preconditioned
conjugate gradient method with implementation issues   is described.
 In section 5,  some numerical experiments are  presented to illustrate the efficiency of both the algorithm for fast matrix-vector multiplications and  the proposed preconditioner. The conclusion is in the last section.

\section{ Preliminaries}
In this section, some properties of  Legendre polynomials and a useful transform are presented.
\subsection{Legendre polynomials}
Denote by $L_n(x)$ the  Legendre polynomial of degree $n$ which satisfies the following three-term recurrence relation:
       \begin{align*}
         L_0(x)&=1, \quad  L_1(x)=x,\\
        (n+1)L_{n+1}(x)&=(2n+1)x L_n(x)-nL_{n-1}(x),  \quad n\ge 1.
       \end{align*}
The  Legendre polynomials are orthogonal to each other with respect to the uniform weight function,
 \begin{align*}
  \int_{-1}^1 L_m(x) L_n(x)\d x =\frac{2}{2m+1}\delta_{mn}, \quad m,n\ge 0,
 \end{align*}
where $\delta_{mn}$ is the Kronecker delta symbol.
Moreover, they satisfy the derivative recurrence relation
 \begin{align}\label{difRecerrence}
 (2n+1)L_n(x)= L'_{n+1}(x)-L'_{n-1}(x), \quad  n\ge 1,
 \end{align}
 and symmetric property
 \begin{align}\label{symproperty}
  L_n(-x)=(-1)^n L_n(x),\quad L_n(\pm 1)=(\pm 1)^n.
 \end{align}

 \begin{lemma}[\cite{Carlitz1961}]
  For $m,n\ge 0,$  it holds that
  \begin{align}\label{LegProduct}
   L_m(x)L_n(x) = \sum_{s=0}^{\min(m,n)}\frac{m+n+\frac{1}{2}-2s}{m+n+\frac{1}{2}-s}\frac{C_s C_{m-s}C_{n-s}}{C_{m+n-s}}
   L_{m+n-2s}(x),
  \end{align}
  where
    $$ C_r =\frac{1\cdot3 \ldots (2r-3)(2r-1)}{r!2^r}.$$
 \end{lemma}

 \subsection{ Discrete Legendre transforms}
Given $N+1$ values $c_0,c_1,\cdots, c_{N}$, the backward discrete Legendre transform (BDLT) calculates the discrete sums:
\begin{align}\label{BDLT}
  f_k = \sum_{n=0}^N c_n L_n(x_k),  \qquad  0\le k\le N,
\end{align}
where  the Legendre-Gauss quadrature nodes $x_0, x_1, \cdots, x_N$ are the roots of $L_{N+1}(x)$.
Given $f_0,f_1,\cdots, f_{N}$, the forward discrete Legendre transform (FDLT) computes $c_0,c_1,\cdots, c_{N}$, which reads
\begin{align}\label{FDLT}
 c_n = \frac{2n+1}{2}\sum_{k=0}^{N}w_k f_k L_n(x_k),  \qquad 0\le n\le N,
\end{align}
where $w_0, w_1, \cdots, w_N$ are the Legendre-Gauss quadrature  weights.
Assuming that $(L_n(x_j))_{j,n=0,\cdots,N}$ have been precomputed, the discrete Legendre transforms \eqref{BDLT}
and \eqref{FDLT} can be carried out by a standard matrix-vector multiplication routine in about $N^2$ flops.
In this paper,  the algorithm in \cite{HaleTownsend2015} is used for the fast computation of the discrete Legendre transforms \eqref{BDLT} and \eqref{FDLT} which is of $\mathcal{O}(N (\log N)^2/\log \log N)$ complexity, because it has no precomputational cost and only involves the FFT and Taylor approximations.

\section{The Legendre-Galerkin method of non-separable second order elliptic equations}

Consider non-separable second order elliptic equations of the form
\begin{align}\label{NSproblem}
\begin{cases}
 -\text{div}( \beta(\bx)\nabla u ) + \alpha(\bx)u = f, \quad\bx\in\Omega=(-1,1)^d,\\
 u|_{\partial \Omega}=0,
\end{cases}
\end{align}
where $d=1,2,3$, the  coefficient functions $\beta(\bx), \alpha(\bx)$  and $f(\bx)$ are continuous, and
 $0<b_1\le \beta(\bx)\le b_2$, $0\le \alpha(\bx)< a$ in $\Omega$ for some positive constants $b_1,b_2,a.$

Let $P_N$ be the space of polynomials of  degree less than or equal to $N$, and
\begin{align*}
 X_N=\{ v\in P_N: v(\pm 1)=0  \}.
\end{align*}
Denote $X^d_N = (X_N)^d$.
Then  the  Legendre-Galerkin approximation to \eqref{NSproblem} is: Find $u_N\in X^d_N$ such that
\begin{align}\label{Discreteproblem}
  (\beta(\bx)\nabla u_N, \nabla v_N )+(\alpha(\bx)u_N, v_N) = (f,v_N), \quad \forall v_N\in X^d_N.
\end{align}
where  $(u,v)=\int_{\Omega}uv \d \bx$ is  the scalar product in  $L^2(\Omega)$.

Denote
\begin{align*}
\phi_k(x):= L_k(x) - L_{k+2}(x).
\end{align*}
Due to \eqref{symproperty}, it is easy to know that the function
$\phi_k(x)$  satisfies the boundary condition of problem \eqref{NSproblem}. Hence,  the
basis functions of the space $X_N$ can be chosen as
$$
\phi_0(x), \phi_1(x), ..., \phi_{N-2}(x).
$$

\begin{itemize}
  \item {\bf{One dimensional case.}}
    Assume
$ u_N= \sum_{n=0}^{N-2} \widehat{u}_n \phi_n(x),$
and denote
\begin{align*}
        &A = \Big[\big(\beta(x)\phi'_k, \phi'_j \big)\Big]_{0\le k,j \le N-2}, \\
        & B = \Big[\big(\alpha(x)\phi_k, \phi_j \big)\Big]_{0\le k,j \le N-2}, \\
        &\widehat{u} = \big( \widehat{u}_0, \widehat{u}_1, \cdots, \widehat{u}_{N-2}  \big)^T,\\
        & F = \Big(  f_0,  f_1, \cdots,  f_{N-2} \Big)^T, \quad   f_k = (f, \phi_k).
\end{align*}
  \item {\bf{Two dimensional case.}}
The multi-dimensional  basis functions are constructed by  using the tensor product of one-dimensional
 basis functions. In two dimensions, they read
 $$ \varphi_{k,j}(\bx): = \phi_k(x)\phi_j(y),\quad k,j = 0,1,\cdots, N-2.$$
Assume
$ u_N= \sum_{k,j=0}^{N-2} \widehat{u}_{k,j} \varphi_{k,j}(\bx),$
and  denote
       \begin{align*}
        &A = \Big[\big(\beta(\bx)\nabla\varphi_{k,j}, \nabla\varphi_{m,n}\big)\Big]_{0\le k,j,m,n \le N-2},\\
        &B = \Big[ \big(\alpha(\bx)\varphi_{k,j}, \varphi_{m,n} \big) \Big]_{0\le k,j,m,n \le N-2},\\
        &\widehat{u} = \Big( \widehat{u}_{0,0}, \widehat{u}_{1,0}, \ldots, \widehat{u}_{N\!-\!2,0};\ldots; \widehat{u}_{0,N\!-\!2}, \widehat{u}_{1,N-2}, \ldots, \widehat{u}_{N-2,N-2}  \Big)^T,\\
        & F = \Big(\, f_{0,0}, f_{1,0}, \cdots, f_{N-2,0};\ldots; f_{0,N-2}, f_{1,N-2}, \ldots, f_{N-2,N-2} \Big)^T, \\
        &f_{k,j} = (f, \varphi_{k,j}).
       \end{align*}
  \item {\bf{Three dimensional case.}} Similarly, the three-dimensional basis functions are as follows
  $$
     \psi_{i,j,k}(\bx): = \phi_i(x)\phi_j(y)\phi_k(z),\quad i,k,j = 0,1,\cdots, N-2.
  $$
    Assume $u_N= \sum_{i,j,k=0}^{N-2} \widehat{u}_{i,j,k} \psi_{i,j,k}(\bx)$, and denote
   \begin{align*}
        &A = \Big[\big(\beta(\bx)\nabla\psi_{i,j,k}, \nabla\psi_{m,n,l}\big)\Big]_{0\le i,k,j,m,n,l \le N-2},\\
        &B = \Big[ \big(\alpha(\bx)\psi_{i,j,k}, \psi_{m,n,l} \big) \Big]_{0\le i,k,j,m,n,l \le N-2},\\
        &\widehat{u} = \big( \widehat{u}_{0}, \widehat{u}_{1},\ldots, \widehat{u}_{N-2}\big)^T,\\
        & \widehat{u}_{k} = \big(\widehat{u}_{0,0,k}, \widehat{u}_{1,0,k}, \ldots, \widehat{u}_{N\!-\!2,0,k};\ldots; \widehat{u}_{0,N\!-\!2,k}, \widehat{u}_{1,N -2,k}, \ldots, \widehat{u}_{N -2,N-2,k}\big),\\
        & F = \big(\, f_{0}, f_{1}, \ldots, f_{N-2} \big)^T,\\
        & f_k = \big( f_{0,0,k}, f_{1,0,k},\, \ldots, \,f_{N-2,0,k};\ldots; f_{0,N-2,k}, f_{1,N-2,k}, \ldots, f_{N-2,N-2,k}\big),\\
        & f_{i,j,k} = (f, \psi_{i,j,k}).
       \end{align*}
\end{itemize}

Then the equation \eqref{Discreteproblem} is equivalent to the following algebraic system:
\begin{align}\label{algebraicsystem}
  ( A + B )\widehat{u} = F.
\end{align}
For variable coefficients $\alpha(\bx)$ and $\beta(\bx)$, the  matrices $A$ and $B$ in equation \eqref{algebraicsystem} are usually dense and ill-conditioned. Hence, it is prohibitive to use a direct inversion method or an iterative method without
preconditioning. Moreover, it is imperative to use an iterative method with a good preconditioner.

\section{ Preconditioned conjugate gradient method }

In this article, the symmetric definite linear system \eqref{algebraicsystem} is solved by the
preconditioned conjugate gradient (PCG) method presented in Algorithm 1.
\begin{algorithm}[h]
\caption{ PCG}
\textbf{Initialize:} $A,B,F$ and initialization vector $x$, preconditioner $M$,
the maximum  loop size $k_{max}$, stop  criteria $\varepsilon$\\
$k=0$, \\
$r = F-(A+B)x$\\
\textbf{While} $\sqrt{r^T r}> \varepsilon \|F\|_2$ \textbf{and} $k<k_{max}$ \textbf{do}\\
\hspace{0.4cm}Solve $Mz = r$,\\
\hspace{0.4cm} $ k=k+1$\\
 \hspace{0.4cm} \textbf{if} $k=1$\\
 \hspace{0.6cm}  $p=z$; $\rho=r^T z$\\
 \hspace{0.4cm}  \textbf{else}\\
 \hspace{0.6cm}   $\tilde{\rho}=\rho$; $\rho=r^T z$; $\beta = \rho/\tilde{\rho};$ $p=z+\beta p$\\
 \hspace{0.4cm}   \textbf{end if} \\
 \hspace{0.4cm}   $w=(A+B)p$; $\alpha=\rho/p^T w$\\
 \hspace{0.4cm}  $x = x+\alpha p;$ $r=r-\alpha w$\\
 \textbf{end while}
\end{algorithm}


 In each iteration of the PCG method, it needs to solve a system with $M$ as the
coefficient matrix and involves a matrix-vector multiplication. The complexity of these two steps dominates
that of the algorithm.
In order to accelerate the convergence rate of conjugate gradient type methods, an efficient preconditioner $M$ is needed.
 Moreover, an accurate numerical solver is needed to solve the preconditioning equation $Mz = r$.
 At last, fast matrix vector multiplications have to be used to further reduce the complexity of the algorithm.

\subsection{Proposed preconditioner}
A preconditioner is  prescribed for the coefficient matrix $A+B$ of the linear system \eqref{algebraicsystem} in the following way.

Since $\beta(\bx)$ and $\alpha(\bx)$ are continuous, they can be approximated by a finite number of Legendre polynomials to any desired accuracy. That is, for any $\epsilon_1 >0, \epsilon_2 >0,$ there  exists  $t_1, t_2\in \NN $ and
$ p_{t_1}\in (P_{t_1})^d$, $ p_{t_2}\in (P_{t_2})^d$
such that
\begin{align}
 & \| \beta(\bx) - p_{t_1} \|_{L^{\infty}([-1,1]^d)} < \epsilon_1, \label{Appbeta}\\
 & \| \alpha(\bx) - p_{t_2} \|_{L^{\infty}([-1,1]^d)} < \epsilon_2   \label{Appalpha}.
\end{align}
It is stressed that $t_1$ and $t_2$ can be surprisingly small when $\beta(\bx), \alpha(\bx)$ are  analytic or many times differentiable. For practical purposes,  the preconditioner $M$ is constructed by  replacing $\beta(\bx)$ with $p_{t_1}$   in the matrix $A$
and  $\alpha(\bx)$ with $p_{t_2}$ in the matrix $B$. More precisely,
\begin{itemize}
\item $d=1$
       \begin{align}
         & p_{t_1}(x)= \sum_{t=0}^{t_1} \widehat{\beta}_t L_t(x), \label{1DMbeta} \\
         & p_{t_2}(x)= \sum_{t=0}^{t_2} \widehat{\alpha}_t L_t(x), \label{1DMalpha} \\
         & M = \Big[\big(p_{t_1}(x)\phi'_i(x), \phi'_j(x) \big) + \big(p_{t_2}(x)\phi_i(x), \phi_j(x) \big) \Big]_{0\le i,j \le N-2} \label{1DM}.
        \end{align}

\item $d=2$
        \begin{align}
          & p_{t_1}(\bx)= \sum_{m,n=0}^{t_1} \widehat{\beta}_{mn} L_m(x)L_n(y), \\
          & p_{t_2}(\bx)= \sum_{m,n=0}^{t_2} \widehat{\alpha}_{mn} L_m(x)L_n(y),\\
          & M = \Big[\big(p_{t_1}(\bx)\nabla\varphi_{i,j}, \nabla\varphi_{m,n} \big)+ \big(p_{t_2}(\bx)\varphi_{i,j}, \varphi_{m,n} \big)
              \Big]_{0\le i,j,m,n \le N-2} \label{2DM}.
        \end{align}

\item $d=3$
        \begin{align}
           & p_{t_1}(\bx)= \sum_{m,n,l=0}^{t_1} \widehat{\beta}_{mnl} L_m(x)L_n(y)L_l(z),\\
           & p_{t_2}(\bx)= \sum_{m,n,l=0}^{t_2} \widehat{\alpha}_{mnl} L_m(x)L_n(y)L_l(z),\\
           &M = \Big[\big(p_{t_1}(\bx)\nabla\psi_{i,j,k}, \nabla\psi_{m,n,l} \big)+ \big(p_{t_2}(\bx)\psi_{i,j,k}, \psi_{m,n,l} \big)
              \Big]_{0\le i,j,k,m,n,l \le N-2}\label{3DM}.
        \end{align}

\end{itemize}

Note that  the  Legendre expansion coefficients in $p_{t_1}(\bx)$ and $ p_{t_2}(\bx)$  can be calculated respectively in  $\mathcal{O}(t_{1}^d (\log t_{1})^2)$ operations and $\mathcal{O}(t_{2}^d (\log t_{2})^2)$ operations  by means of the fast discrete Legendre transform.
Moreover, the  preconditioner $M$  for one-dimensional case is a  banded matrix with a fixed bandwidth dependent of $t_1$ and $t_2$ from the following proposition.

\begin{proposition}\label{Matrixproposition}
 Denote
 \begin{align*}
  M_1 &= \big[\beta_{ji} \big]_{0\le i,j \le N-2},\quad \beta_{ji}=\big(p_{t_1}(x)\phi'_i(x), \phi'_j(x) \big),\\
  M_2 &= \big[\alpha_{ji} \big]_{0\le i,j \le N-2},\quad \alpha_{ji}=\big(p_{t_2}(x)\phi_i(x), \phi_j(x) \big).
 \end{align*}
 If $t_1$  in \eqref{1DMbeta}  and $t_2$ in \eqref{1DMalpha} are fixed, the bandwidth $q_1$ of banded matrix $M_1$  and the bandwidth $q_2$ of banded matrix $M_2$
 are as follows:
 \begin{align*}
 & q_1=
  \begin{cases}
   2t_1 -1, & t_1 \, \text{even},\\
   2t_1 +1, & t_1 \, \text{odd},
  \end{cases}
  \qquad \beta(x) \text{ is\, an\, odd\, function;}\\
 &  q_1=
  \begin{cases}
   2t_1 +1, & t_1 \, \text{even},\\
   2t_1 -1, & t_1 \, \text{odd},
  \end{cases}
  \qquad \beta(x) \text{ is\, an\, even\, function;}\\
 &  q_2=
  \begin{cases}
   2t_2 +3, & t_1 \, \text{even},\\
   2t_2 +5, & t_1 \, \text{odd},
  \end{cases}
  \qquad \alpha(x) \text{ is\, an\, odd\, function;}\\
 &  q_2=
  \begin{cases}
   2t_2 +5, & t_2 \, \text{even},\\
   2t_2 +3, & t_2 \, \text{odd},
  \end{cases}
  \qquad \alpha(x) \text{ is\, an\,even\, function.}
 \end{align*}

\end{proposition}
\begin{proof}
 In the case $d$=1, it follows from \eqref{LegProduct} that  both $p_{t_1}(x)\phi'_k(x)$ and $p_{t_2}(x)\phi_k(x)$ can be  represented in Legendre series, i.e.,
 \begin{align*}
   &p_{t_1}(x)\phi'_k(x)=(-2k-3)L_{k+1}(x)\sum_{t=0}^{t_1} \widehat{\beta}_t L_t(x)=
   \begin{cases}
     \sum_{j=0}^{t_1 + k+1}\tilde{\beta}_j L_j,       &t_1 \ge k+1,\\
     \sum_{j=k+1-t_1}^{t_1 + k+1}\tilde{\beta}_j L_j, &t_1 < k+1,
   \end{cases}\\
  & p_{t_2}(x)\phi_k(x)=(L_{k}(x)-L_{k+2}(x))\sum_{t=0}^{t_2} \widehat{\alpha}_t L_t(x)=
   \begin{cases}
     \sum_{j=0}^{t_2 + k +2}\tilde{\alpha}_j L_j,       &t_2 \ge k,\\
     \sum_{j=k-t_2}^{t_2 + k+2}\tilde{\alpha}_j L_j,    &t_2 < k,
   \end{cases}
 \end{align*}
where $\tilde{\beta}_j$  are  Legendre expansion coefficients in terms of  $C_j$ in \eqref{LegProduct} and $\widehat{\beta}_j$, $\tilde{\alpha}_j$ are  Legendre expansion coefficients in terms of  $C_j$ in \eqref{LegProduct} and $\widehat{\alpha}_j$.
Together with  parity arguments on $\beta(x)$ and $\alpha(x)$, this leads to the conclusion.
\end{proof}

\begin{remark}
The bandwidth of  preconditioner $M$ is of $\mathcal{O}(N)$  in the case $d=2$ and of $\mathcal{O}(N^2)$ in the case $d=3$.
\end{remark}

\subsection{ Incomplete LU preconditioning for banded linear systems}

Without loss of generality, assume $\alpha(\bx)=0$ in problem \eqref{NSproblem}.
And the preconditioner $M$ is constructed by using a ($T$+1)-term Legendre series in each direction to approximate the coefficient function $\beta(\bx)$.
In what follows, it is  shown that the preconditioning equation $Mz=r$ is approximately solved
by a one-step iterative process based on the ILU(0) factorization, see for instance, \cite[Chapter 10]{saad2000},  in $\mathcal{O}(T^{2d} N^d)$ operations for $d=1,2,3$.

To approximately solve $Mz=r$, proceed as follows:
\begin{enumerate}[Step 1.]
 \item Perform the ILU(0) factorization to obtain a sparse unit lower triangular matrix $L$ and a sparse upper triangular matrix $U$;
 \item Solve the unit lower triangular system $Ly=r$ by a forward substitution shown in Algorithm 3;
 \item Solve the upper triangular system $Uz=y$ by a backward substitution shown in Algorithm 4.

\end{enumerate}

For the sparse matrix $M$ whose elements are $m_{ij},i,j = 1,\ldots,(N-1)^d$,
the incomplete LU factorization process with no fill-in, denoted by ILU(0), is to compute a sparse unit lower triangular matrix $L$ and a sparse upper triangular matrix $U$ so that
the elements of $M-LU$ are zeros in the locations of $NZ(M)$, where  $NZ(M)$ is the set of pairs $(i,j), 1\le i,j\le (N-1)^d$ such that $m_{ij}\neq 0,$ and the entries in
the extra diagonals in the product $LU$ are called fill-in elements. Due to the fact that
fill-in elements are ignored,  it is possible to find $L$ and $U$ so that their product is equal to $M$ in the other diagonals.
By definition, together the $L$ and $U$ matrices in ILU(0) have the same number of nonzero elements as the matrix $M$.

\begin{algorithm}[h]
\caption{ILU(0)}
\textbf{Initialize:}  Given $H\in \RR^{n\times n}$, the following algorithm computes an unit lower triangular matrix $L$ and an upper triangular matrix $U$, assuming they exist. $H(i,j)$ is overwritten by $L(i,j)$ if $i>j$ and by $U(i,j)$ otherwise.\\
\textbf{for} $i = 2:n$\\
\hspace{0.5cm} \textbf{for} $k = 1:i-1$ and $(i,k)\in NZ(H)$\\
\hspace{1cm}   $ H(i,k) = H(i,k)/H(k,k)$\\
\hspace{1cm} \textbf{for} $j = k+1:n$ and $(i,j)\in NZ(H)$\\
 \hspace{1.6cm}   $H(i,j) = H(i,j)-H(i,k)\cdot H(k,j)$\\
 \hspace{1cm}  \textbf{end for} \\
 \hspace{0.5cm}  \textbf{end for} \\
 \textbf{end for}
\end{algorithm}

\begin{algorithm}[h]
\caption{ Forward substitution}
\textbf{Initialize:}  Given an unit lower triangular matrix $L\in \RR^{n\times n}$ and a vector $r\in \RR^{n}$, the following algorithm computes the linear system $Ly = r$.\\
$y(1) = r(1)$\\
\textbf{for} $i= 2:n$\\
\hspace{0.5cm} \textbf{for} $j = 1:i-1$ and $(i,j)\in NZ(L)$\\
\hspace{1cm}    $y(i) = y(i) + L(i,j)\cdot y(j)$\\
\hspace{0.5cm}  \textbf{end for} \\
\hspace{0.5cm}  $y(i) = r(i)-y(i)$\\
 \textbf{end for}
\end{algorithm}

\begin{algorithm}[h]
\caption{ Backward substitution}
\textbf{Initialize:}  Given an upper triangular matrix  $U\in \RR^{n\times n}$ and a vector $y\in \RR^{n}$, the following algorithm computes the linear system $Uz = y$.\\
$z(n)= y(n)/U(n,n)$\\
\textbf{for} $i = n-1:1$\\
\hspace{0.5cm} \textbf{for} $j = n:i$ and $(i,j)\in NZ(U)$\\
\hspace{1cm}  $y(i)= y(i)- U(i,j)\cdot z(j)$\\
\hspace{0.5cm} \textbf{end for} \\
\hspace{0.5cm} $z(i) = y(i)/U(i,i)$\\
\textbf{end for}
\end{algorithm}


 To evaluate the complexity of the one-step process to solve a system with the coefficient matrix $M$, the  number of nonzeros in $M$ is considered.
Taking two-dimensional problems as an example, the matrix $M$ can be rewritted in the following formulation:
 \begin{align*}
  & p_{T}(\bx)= \sum_{t=0}^{T}\sum_{k=0}^{T} \widehat{\beta}_{tk} L_t(x)L_k(y),\\
  & M = \sum_{t=0}^{T}\sum_{k=0}^{T} \widehat{\beta}_{tk}
   \Big[ M^{(k)}_{1y}\otimes S^{(t)}_{1x} + S^{(k)}_{1y}\otimes M^{(t)}_{1x} \Big],
 \end{align*}
where
\begin{align*}
 & M^{(k)}_{1y}= \Big[ \big(L_k(y)\phi_j(y), \phi_n(y)  \big) \Big]_{0\le j,n\le N-1},\\
 & S^{(t)}_{1x}= \Big[ \big(L_t(x)\phi'_i(x), \phi'_m(x)  \big) \Big]_{0\le i,m\le N-1},\\
 & S^{(k)}_{1y}= \Big[ \big(L_k(y)\phi'_j(y), \phi'_n(y)  \big) \Big]_{0\le j,n\le N-1},\\
 & M^{(t)}_{1x}= \Big[ \big(L_t(x)\phi_i(x), \phi_m(x)  \big) \Big]_{0\le i,m\le N-1}.
\end{align*}
It follows from Proposition \ref{Matrixproposition} that  each matrix $M^{(T)}_{1y}, S^{(T)}_{1x}, S^{(T)}_{1y}, M^{(T)}_{1x}$ has $\mathcal{O}(T N)$ nonzero elements.
Thus, the number of nonzeros in $M$ is of $\mathcal{O}(T^2N^2)$.
Then it is deduced that the number of nonzeros in $M$ for three-dimensional problems is of $\mathcal{O}(T^3 N^3)$.

From Algorithm 2, the cost of performing the ILU(0) factorization essentially  depends on the number of nonzero elements in $M$, which is of $\mathcal{O}(T^{2d} N^d)$. And the complexity of performing either the forward substitution in Algorithm 3 or the backward substitution in Algorithm 4 is of $\mathcal{O}(T^{d} N^d)$. As a result, the one-step iterative process to approximately solve the preconditioning equation $Mz=r$  costs  $\mathcal{O}(T^{2d} N^d)$  numerical  operations, $d=1,2,3$.


\subsection{Fast matrix-vector multiplications}

The fast transforms of the Legendre expansions provide the possibility for fast matrix-vector multiplications of  vectors by the discrete matrix $A+B$ resulting from the Legendre-Gelerkin method.

Denote $\Lambda=(a,b)^d$ and $P_N^d=(P_N)^d$. Define the interpolation operator $I_N: C(\Lambda)\rightarrow P_N^d(\Lambda)$ such that for any $u\in C(\Lambda)$,
\begin{align*}
  (I_N u)(\bx) = u(\bx),  \quad \bx\in \{ x_0, x_1, \cdots, x_N \}^d,
\end{align*}
where $x_0, x_1, \cdots, x_N$ are  the Legendre-Gauss quadrature nodes mentioned above.

Given the coefficient vector $p$ of $u_N\in X_N^d,$  the matrix-vector multiplicaton of $(A+B)p$ is performed as follows
(with the operation counts of each step in parenthese):

\begin{enumerate}[Step 1.]
 \item  Compute the Legendre coefficients of $\nabla u_N $ and $u_N$ respectively; \big($\mathcal{O}(N^d)$\big)
 \item  Perform the BDLT of $\nabla u_N $ and $u_N$ respectively;  \big($\mathcal{O}(N^d (\log N)^2)$ \big)
 \item  Compute $\beta(\bx)\nabla u_N$ and $\alpha(\bx) u_N$ at the Legendre-Gauss quadrature nodes and then  the FDLT  of  $I_N(\beta(\bx)\nabla u_N)$, $I_N(\alpha(\bx) u_N)$;  \big($\mathcal{O}(N^d (\log N)^2)$ \big)
 \item  Compute the matrix-vector multiplicaton of $(A+B)p$.  \big($\mathcal{O}(N^d)$ \big)
\end{enumerate}

For clarity of presentation, fast  matrix-vector multiplicatons are described in  details.

{\bf{One dimensional case.}} Given $u_N = \sum_{k=0}^{N-2}\widehat{u}_k \phi_k(x)$, the computation
\begin{align*}
 (A \widehat{u})_j = \big( I_N(\beta u'_N), \phi'_j \big),  \qquad j=0,1,\cdots, N-2
\end{align*}
without explicitly forming the matrix $A$ is presented as follows.
\begin{itemize}
\item[1,] Using \eqref{difRecerrence} to determine $\{ \widetilde{u}' \}$ from
  \begin{align*}
   u'_N(x) = \sum_{k=0}^{N-2}\widehat{u}_k\phi'_k(x) = \sum_{k=0}^{N}\widetilde{u}'_k L_k(x);
  \end{align*}
\item[2,] (BDLT) Compute
  \begin{align*}
   u'_N(x_j) = \sum_{k=0}^{N}\widetilde{u}'_k L_k(x_j), \qquad j=0,1,\cdots, N;
  \end{align*}
\item[3,] (FDLT) Determine $\{ \widehat{\beta}_k\}$ from
  \begin{align*}
   I_N (\beta u'_N)(x_j) =\sum_{k=0}^{N} \widehat{\beta}_k L_k(x_j), \qquad j=0,1,\cdots, N;
  \end{align*}
\item[4,] For $j=0,1,\cdots, N-2,$ compute
  \begin{align*}
   (A \widehat{u})_j = \Big( I_N (\beta u'_N), -(2j+3)L_{j+1}(x) \Big)= -2\widehat{\beta}_{j+1}.
  \end{align*}
\end{itemize}

Similarly, the computation
\begin{align*}
 (B \widehat{u})_j = \big( I_N(\alpha u_N), \phi_j \big),  \qquad j=0,1,\cdots, N-2
\end{align*}
without explicitly forming the matrix $B$ is presented as follows.
\begin{itemize}
\item[1,] Determine $\{ \widehat{u}_k^{(1)}\}$ from
    \begin{align*}
     u_N(x) = \sum_{k=0}^{N-2}\widehat{u}_k \phi_k(x)=\sum_{k=0}^{N}\widehat{u}_k^{(1)} L_k(x);
    \end{align*}
 \item[2,](BDLT) Compute
    \begin{align*}
     u_N(x_j) = \sum_{k=0}^{N}\widehat{u}_k^{(1)} L_k(x_j), \qquad j=0,1,\cdots, N;
    \end{align*}
 \item[3,] (FDLT) Determine $\{ \widehat{\alpha}_k\}$ from
     \begin{align*}
      I_N (\alpha u_N)(x_j) =\sum_{k=0}^{N} \widehat{\alpha}_k L_k(x_j), \qquad j=0,1,\cdots, N;
     \end{align*}
 \item[4,] For $j=0,1,\cdots, N-2,$ compute
  \begin{align*}
   (B \widehat{u})_j = \Big( I_N (\alpha u_N), \phi_{j}(x) \Big)
   = \frac{2\widehat{\alpha}_j}{2j+1}- \frac{2\widehat{\alpha}_{j+2}}{2j+5} .
  \end{align*}

\end{itemize}

{\bf{Two dimensional case.}} Given $ u_N= \sum_{k,j=0}^{N-2} \widehat{u}_{kj} \varphi_{k,j}(\bx)$, the calculation
\begin{align*}
  (A \widehat{u})_{kj}  = \big(I_N(\beta\nabla u_N), \nabla\varphi_{k,j}\big),\qquad k,j=0,1,\cdots, N-2
\end{align*}
without explicitly forming the matrix $A$ is presented below.
\begin{itemize}
 \item[1,] Using \eqref{difRecerrence} to determine $\{ \widetilde{u}^{x}_{kj} \}$ and $\{ \widetilde{u}^{y}_{kj} \}$   from
   \begin{align*}
     &\partial_x u_N = \sum_{k=0}^{N-2}\sum_{j=0}^{N-2}\widehat{u}_{kj}\phi'_{k}(x)\phi_j(y)
                    = \sum_{k=0}^{N}\sum_{j=0}^{N} \widetilde{u}^{x}_{kj}L_k(x)L_j(y),\\
     &\partial_y u_N = \sum_{k=0}^{N-2}\sum_{j=0}^{N-2}\widehat{u}_{kj}\phi_k(x)\phi'_{j}(y)
                      = \sum_{k=0}^{N}\sum_{j=0}^{N} \widetilde{u}^y_{kj}L_k(x)L_j(y);
   \end{align*}
\item[2,] (BDLT) For $m,n=0,1,\cdots, N$, compute
 \begin{align*}
   &\partial_x u_N(x_m,y_n)= \sum_{k=0}^{N}\sum_{j=0}^{N} \widetilde{u}^{x}_{kj}L_k(x_m)L_j(y_n),\\
   &\partial_y u_N(x_m,y_n)= \sum_{k=0}^{N}\sum_{j=0}^{N} \widetilde{u}^y_{kj}L_k(x_m)L_j(y_n);
 \end{align*}

\item[3,] (FDLT) Determine $\{ \widehat{\beta}^{x}_{kj}\}$ and $\{ \widehat{\beta}^{y}_{kj}\}$ from
  \begin{align*}
   & I_N (\beta\partial_x u_N)(x_m,y_n)=\sum_{k=0}^{N}\sum_{j=0}^{N} \widehat{\beta}^{x}_{kj}L_k(x_m)L_j(y_n),\\
   & I_N (\beta\partial_y u_N)(x_m,y_n)=\sum_{k=0}^{N}\sum_{j=0}^{N} \widehat{\beta}^{y}_{kj}L_k(x_m)L_j(y_n);
  \end{align*}
  \item[4,] For $k,j=0,1,\cdots, N-2,$ compute
    \begin{align*}
     &(A \widehat{u})_{kj}
      = \Big( I_N(\beta\partial_x u_N), \phi'_k(x)\phi_j(y)\Big)+ \Big( I_N(\beta\partial_y u_N), \phi_k(x)\phi'_j(y)\Big)\\
      &= \Big( I_N(\beta\partial_x u_N), (-2k-3)L_{k+1}(x)\phi_j(y)\Big)\!+\! \Big( I_N(\beta\partial_y u_N), (-2j-3)\phi_k(x)L_{j+1}(y)\Big)\\
      &= \frac{4}{2j+5}\widehat{\beta}^{x}_{k+1,j+2}-\frac{4}{2j+1}\widehat{\beta}^{x}_{k+1,j}
         +\frac{4}{2k+5}\widehat{\beta}^{y}_{k+2,j+1}-\frac{4}{2k+1}\widehat{\beta}^{y}_{k,j+1}.
    \end{align*}
\end{itemize}

 Similarly, the evaluation
\begin{align*}
  (B \widehat{u})_{kj}  = \big( I_N(\alpha u_N), \varphi_{k,j}\big),\qquad k,j=0,1,\cdots, N-2
\end{align*}
without explicitly forming the matrix $B$ is shown below.
\begin{itemize}
 \item[1,] Determine $\{ \widetilde{u}_{kj}^{(1)} \}$ from
   \begin{align*}
      u_N = \sum_{k=0}^{N-2}\sum_{j=0}^{N-2}\widehat{u}_{kj}\phi_{k}(x)\phi_j(y)
                    = \sum_{k=0}^{N}\sum_{j=0}^{N}\widetilde{u}_{kj}^{(1)} L_k(x)L_j(y);
   \end{align*}
 \item[2,](BDLT) Compute
    \begin{align*}
     u_N(x_m,y_m) = \sum_{k=0}^{N}\sum_{j=0}^{N}\widetilde{u}_{kj}^{(1)} L_k(x_m)L_j(y_n), \qquad m,n=0,1,\cdots, N;
    \end{align*}

 \item[3,](FDLT) Determine $\{ \widehat{\alpha}_{kj}\}$ from
     \begin{align*}
      I_N (\alpha u_N)(x_m,y_n) =\sum_{k=0}^{N}\sum_{j=0}^{N} \widehat{\alpha}_{kj} L_k(x_m)L_j(y_n), \qquad m,n=0,1,\cdots, N;
     \end{align*}

 \item[4,]For $k,j=0,1,\cdots, N-2,$ compute
  \begin{align*}
   &(B \widehat{u})_{kj} = \Big( I_N (\alpha u_N), \phi_k(x)\phi_j(y) \Big)\\
   &= \frac{4\widehat{\alpha}_{k,j}}{(2k+1)(2j+1)}- \frac{4\widehat{\alpha}_{k,j+2}}{(2k+1)(2j+5)}
      - \frac{4\widehat{\alpha}_{k+2,j}}{(2k+5)(2j+1)}+\frac{4\widehat{\alpha}_{k+2,j+2}}{(2k+5)(2j+5)}  .
  \end{align*}

\end{itemize}

{\bf{Three dimensional case.}} Given $ u_N= \sum_{k,j,l=0}^{N-2} \widehat{u}_{kjl} \psi_{k,j,l}(\bx)$, the evaluation
\begin{align*}
  (A \widehat{u})_{kjl}  = \big( I_N(\beta\nabla u_N), \nabla\psi_{k,j,l}\big),\qquad k,j,l=0,1,\cdots, N-2
\end{align*}
without explicitly forming the matrix $A$ is presented below.
\begin{itemize}
 \item[1,] Using \eqref{difRecerrence} to determine $\{ \widetilde{u}^{x}_{kjl} \}$, $\{ \widetilde{u}^{y}_{kjl} \}$ and $\{ \widetilde{u}^{z}_{kjl} \}$ from
   \begin{align*}
     &\partial_x u_N = \sum_{k=0}^{N-2}\sum_{j=0}^{N-2}\sum_{l=0}^{N-2}\widehat{u}_{kjl}\phi'_{k}(x)\phi_j(y)\phi_l(z)
            = \sum_{k=0}^{N}\sum_{j=0}^{N}\sum_{l=0}^{N} \widetilde{u}^{x}_{kjl}L_k(x)L_j(y)L_l(z),\\
     &\partial_y u_N = \sum_{k=0}^{N-2}\sum_{j=0}^{N-2}\sum_{l=0}^{N-2}\widehat{u}_{kjl}\phi_k(x)\phi'_{j}(y)\phi_l(z)
             = \sum_{k=0}^{N}\sum_{j=0}^{N}\sum_{l=0}^{N} \widetilde{u}^{y}_{kjl}L_k(x)L_j(y)L_l(z),\\
     &\partial_z u_N = \sum_{k=0}^{N-2}\sum_{j=0}^{N-2}\sum_{l=0}^{N-2}\widehat{u}_{kjl}\phi_k(x)\phi_{j}(y)\phi'_l(z)
             = \sum_{k=0}^{N}\sum_{j=0}^{N}\sum_{l=0}^{N} \widetilde{u}^{z}_{kjl}L_k(x)L_j(y)L_l(z);
   \end{align*}
\item[2,] (BDLT) For $m,n,i=0,1,\cdots, N$, compute
 \begin{align*}
   &\partial_x u_N(x_m,y_n,z_i)= \sum_{k=0}^{N}\sum_{j=0}^{N}\sum_{l=0}^{N} \widetilde{u}^{x}_{kjl} L_k(x_m)L_j(y_n)L_l(z_i),\\
   &\partial_y u_N(x_m,y_n, z_i )=\sum_{k=0}^{N}\sum_{j=0}^{N}\sum_{l=0}^{N} \widetilde{u}^{y}_{kjl} L_k(x_m)L_j(y_n)L_l(z_i),\\
   &\partial_z u_N(x_m,y_n,z_i)= \sum_{k=0}^{N}\sum_{j=0}^{N}\sum_{l=0}^{N} \widetilde{u}^{z}_{kjl} L_k(x_m)L_j(y_n)L_l(z_i);
 \end{align*}

\item[3,] (FDLT) Determine $\{ \widehat{\beta}^{x}_{kjl}\}$, $\{ \widehat{\beta}^{y}_{kjl}\}$ and $\{ \widehat{\beta}^{z}_{kjl}\}$  from
  \begin{align*}
   & I_N (\beta\partial_x u_N)(x_m,y_n,z_i)=\sum_{k=0}^{N}\sum_{j=0}^{N}\sum_{l=0}^{N} \widehat{\beta}^{x}_{kjl}L_k(x_m)L_j(y_n)L_l(z_i),\\
   & I_N (\beta\partial_y u_N)(x_m,y_n,z_i)=\sum_{k=0}^{N}\sum_{j=0}^{N}\sum_{l=0}^{N} \widehat{\beta}^{y}_{kjl}L_k(x_m)L_j(y_n)L_l(z_i),\\
   & I_N (\beta\partial_z u_N)(x_m,y_n,z_i)=\sum_{k=0}^{N}\sum_{j=0}^{N}\sum_{l=0}^{N} \widehat{\beta}^{z}_{kjl}L_k(x_m)L_j(y_n)L_l(z_i);
  \end{align*}
  \item[4,] For $k,j,l=0,1,\cdots, N-2,$ compute
    \begin{align*}
     (A \widehat{u})_{kjl}
     =& \Big( I_N(\beta\partial_x u_N), \phi'_k(x)\phi_j(y)\phi_l(z) \Big)
       + \Big( I_N(\beta\partial_y u_N), \phi_k(x)\phi'_j(y)\phi_l(z)\Big)\\
      & + \Big( I_N(\beta\partial_z u_N), \phi_k(x)\phi_j(y)\phi'_l(z)\Big)\\
      =& -\frac{8}{(2j+1)(2l+1)}\widehat{\beta}^{x}_{k+1,j,l}+\frac{8}{(2j+5)(2l+1)}\widehat{\beta}^{x}_{k+1,j+2,l}\\
       & +\frac{8}{(2j+1)(2l+5)}\widehat{\beta}^{x}_{k+1,j,l+2}-\frac{8}{(2j+5)(2l+5)}\widehat{\beta}^{x}_{k+1,j+2,l+2}\\
       &-\frac{8}{(2k+1)(2l+1)}\widehat{\beta}^{y}_{k,j+1,l}+ \frac{8}{(2k+5)(2l+1)}\widehat{\beta}^{y}_{k+2,j+1,l}\\
        &+\frac{8}{(2k+1)(2l+5)}\widehat{\beta}^{y}_{k,j+1,l+2}-\frac{8}{(2k+5)(2l+5)}\widehat{\beta}^{y}_{k+2,j+1,l+2}\\
        &-\frac{8}{(2k+1)(2j+1)}\widehat{\beta}^{z}_{k,j,l+1}+ \frac{8}{(2k+5)(2j+1)}\widehat{\beta}^{z}_{k+2,j,l+1}\\
        & +\frac{8}{(2k+1)(2j+5)}\widehat{\beta}^{z}_{k,j+2,l+1}-\frac{8}{(2k+5)(2j+5)}\widehat{\beta}^{z}_{k+2,j+2,l+1};
    \end{align*}
\end{itemize}

Similarly, the computation
\begin{align*}
  (B \widehat{u})_{kjl}  = \big(I_N(\alpha u_N), \psi_{k,j,l}\big),\qquad k,j,l=0,1,\cdots, N-2
\end{align*}
without explicitly forming the matrix $B$ is shown below.
\begin{itemize}
 \item[1,] Determine $\{ \widetilde{u}_{kj}^{(1)} \}$ from
   \begin{align*}
      u_N = \sum_{k=0}^{N-2}\sum_{j=0}^{N-2}\sum_{l=0}^{N-2}\widehat{u}_{kjl}\phi_{k}(x)\phi_j(y)\phi_l(z)
                    = \sum_{k=0}^{N}\sum_{j=0}^{N}\sum_{l=0}^{N}\widetilde{u}_{kjl}^{(1)} L_k(x)L_j(y)L_l(z);
   \end{align*}
 \item[2,](BDLT) Compute
    \begin{align*}
     u_N(x_m,y_n,z_i) = \sum_{k=0}^{N}\sum_{j=0}^{N}\sum_{l=0}^{N}\widetilde{u}_{kjl}^{(1)} \phi_k(x_m)\phi_j(y_n)\phi_l(z_i), \qquad m,n,i=0,1,\cdots, N;
    \end{align*}

 \item[3,](FDLT) Determine $\{ \widehat{\alpha}_{kjl}\}$ from
     \begin{align*}
      I_N (\alpha u_N)(x_m,y_n,z_i) =\sum_{k=0}^{N}\sum_{j=0}^{N}\sum_{l=0}^{N} \widehat{\alpha}_{kjl} L_k(x_m)L_j(y_n)L_l(z_i), \quad m,n,i=0,1,\cdots, N;
     \end{align*}

 \item[4,]For $k,j,l=0,1,\cdots, N-2,$ compute
  \begin{align*}
   (B& \widehat{u})_{kjl} = \Big( I_N (\alpha u_N), \phi_k(x)\phi_j(y)\phi_l(z) \Big)\\
   =& -\frac{8\widehat{\alpha}_{k,j,l+2}}{(2k+1)(2j+1)(2l+5)}+\frac{8\widehat{\alpha}_{k,j,l}}{(2k+1)(2j+1)(2l+1)}
     - \frac{8\widehat{\alpha}_{k,j+2,l}}{(2k+1)(2j+5)(2l+1)}\\
    &+\frac{8\widehat{\alpha}_{k,j+2,l+2}}{(2k+1)(2j+5)(2l+5)}-\frac{8\widehat{\alpha}_{k+2,j,l}}{(2k+5)(2j+1)(2l+1)}
     + \frac{8\widehat{\alpha}_{k+2,j+2,l}}{(2k+5)(2j+5)(2l+1)}\\
    & + \frac{8\widehat{\alpha}_{k+2,j,l+2}}{(2k+5)(2j+1)(2l+5)}- \frac{8\widehat{\alpha}_{k+2,j+2,l+2}}{(2k+5)(2j+5)(2l+5)}.
  \end{align*}

\end{itemize}

Note that the main cost in the above procedure of evaluating $A\widehat{u}$ and $B\widehat{u}$ is  the  discrete Legendre transforms in steps 2 and 3. The cost for each of steps 1 and 4 is of $\mathcal{O}(N^d)$ flops.  In summary, the total
cost for evaluating $(A+B)\widehat{u}$ is dominated by several
fast discrete legendre transforms, and is of $\mathcal{O}(N^d(\log N)^2)$.


\section{Numerical results}

In  this section, some numerical  experiments are provided to  demonstrate  the effectiveness of both  the proposed preconditioner $M$   and matrix-vector multiplications. Meanwhile, the properties of matrices from the Legendre-Galerkin methods are numerically studied.
In particular, a class of coefficient functions  with high variations are test.
In all numerical experiments, the stopping criterion $\varepsilon=10^{-12}$.
All the numerical  results  are performed on a 3.30GHz Intel Core i5-4590 desktop computer with 12GB RAM.  The code is in MATLAB 2016b.

\subsection{ Numerical  results for fast matrix-vector multiplications}

The first test investigates the time taken to compute one matrix-vector multiplication of a vector by the discretization matrix resulting from the Legendre-Gelerkin method. The vector is generated randomly by the  rand() command.
 For this purpose,   the numerical experiments are carried out for  different coefficients $\beta(\bx)$ and $\alpha(\bx)$  in one, two and three dimensions.

\begin{table}[htp]
\centering
\caption{CPU time for fast multiplication of matrix $A\in \RR^{(N-1)\times(N-1)}$ by any vector.}
\vspace*{2pt}

\begin{tabular}{|p{1.3cm}<{\centering}|p{1.3cm}<{\centering}|p{1.3cm}<{\centering}|p{1.3cm}<{\centering}|p{1.3cm}<{\centering}|p{1.3cm}<{\centering}|p{1.3cm}<{\centering}|}
 \hline
  \multicolumn{7}{|c|}{  $\beta(x)= ( 2x^2+1)^4$} \\[4pt]
 \hline
  $N$  &  320  & 640 &  1280 & 2560 & 5120  &10240  \\
 \hline
  time(s) &0.0157& 0.0211& 0.0306 &  0.0514 & 0.1130 & 0.2349 \\
 \hline
 \hline
  \multicolumn{7}{|c|}{  $\beta(x)= e^{2x}$} \\[4pt]
 \hline
  $N$  &  320  & 640 & 1280 & 2560 & 5120 &10240  \\
 \hline
  time(s)  &0.0184& 0.0207& 0.0301 & 0.0522 &0.1128  &0.2376 \\
 \hline
\end{tabular}
\label{MVTime1D}
\end{table}

\begin{table}[htp]
\centering
\caption{CPU time for fast multiplication of matrix $B\in \RR^{(N-1)\times(N-1)}$ by any vector.}
\vspace*{2pt}

\begin{tabular}{|p{1.3cm}<{\centering}|p{1.3cm}<{\centering}|p{1.3cm}<{\centering}|p{1.3cm}<{\centering}|p{1.3cm}<{\centering}|p{1.3cm}<{\centering}|p{1.3cm}<{\centering}|}
 \hline
  \multicolumn{7}{|c|}{ $\alpha(x)= ( 2x^2+1)^4$ } \\[4pt]
 \hline
  $N$  &  320  & 640 &  1280 & 2560 & 5120  &10240  \\
 \hline
  time(s) &0.0178& 0.0215& 0.0320 &  0.0529 &0.1140  &0.2414 \\
 \hline
 \hline
  \multicolumn{7}{|c|}{ $\alpha(x)= e^{2x}$ } \\[4pt]
 \hline
  $N$  &  320  & 640 & 1280 & 2560 & 5120 &10240  \\
 \hline
  time(s)  &0.0202& 0.0236& 0.0300 & 0.0518 & 0.1188  & 0.2404 \\
 \hline
\end{tabular}
\label{MVTime1DB}
\end{table}

\begin{table}[htp]
\centering
\caption{CPU time  for fast multiplication of matrix $A\in \RR^{(N\!-\!1)^2\times(N\!-\!1)^2}$ by any vector.}
\vspace*{2pt}

\begin{tabular}{|p{1.3cm}<{\centering}|p{1.3cm}<{\centering}|p{1.3cm}<{\centering}|p{1.3cm}<{\centering}|p{1.3cm}<{\centering}|p{1.3cm}<{\centering}|p{1.3cm}<{\centering}|}
 \hline
  \multicolumn{7}{|c|}{  $\beta(\bx)= \big( 2x^2 + 2y^2 +1\big)^4$ } \\[4pt]
 \hline
  $N$     & 16 & 32 & 64  &128       & 256     & 512 \\
 \hline
  time(s) & 0.0596 & 0.2052  &0.7670 & 2.9913 &12.2643 & 50.4063\\
 \hline
 \hline
  \multicolumn{7}{|c|}{  $\beta(\bx)= e^{2(x+ y)}$ } \\[4pt]
 \hline
  $N$  &  16  & 32 & 64 & 128 & 256 &512  \\
 \hline
  time(s)  &0.0575& 0.2077& 0.7650 & 2.9453 &12.0757  & 49.7519 \\
 \hline
\end{tabular}
\label{MVTime2D}
\end{table}

\begin{table}[htp]
\centering
\caption{CPU time  for fast multiplication of matrix $B\in \RR^{(N\!-\!1)^2\times(N\!-\!1)^2}$ by any vector.}
\vspace*{2pt}

\begin{tabular}{|p{1.3cm}<{\centering}|p{1.3cm}<{\centering}|p{1.3cm}<{\centering}|p{1.3cm}<{\centering}|p{1.3cm}<{\centering}|p{1.3cm}<{\centering}|p{1.3cm}<{\centering}|}
 \hline
  \multicolumn{7}{|c|}{  $\alpha(\bx)= \big( 2x^2 + 2y^2 +1\big)^4$ } \\[4pt]
 \hline
  $N$     & 16 & 32 & 64  &128       & 256     & 512 \\
 \hline
  time(s) & 0.0297 & 0.1046 &0.3723 & 1.5264 &6.2888 &26.0985\\
 \hline
 \hline
  \multicolumn{7}{|c|}{  $\alpha(\bx)= e^{2(x+ y)}$ } \\[4pt]
 \hline
  $N$  &  16  & 32 & 64 & 128 & 256 &512  \\
 \hline
  time(s)  &0.0281 &0.1049  & 0.3794 & 1.4797 &6.0668  &25.1772 \\
 \hline
\end{tabular}
\label{MVTime2DB}
\end{table}

\begin{table}[htp]
\centering
\caption{CPU time  for fast multiplication of matrix $A\in \RR^{(N\!-\!1)^3\times(N\!-\!1)^3}$ by any vector.}
\vspace*{2pt}

\begin{tabular}{|p{1.3cm}<{\centering}|p{1.3cm}<{\centering}|p{1.3cm}<{\centering}|p{1.3cm}<{\centering}|p{1.3cm}<{\centering}|p{1.3cm}<{\centering}|}
 \hline
  \multicolumn{6}{|c|}{  $\beta(\bx)= \big( 2x^2 + 2y^2 + 2z^2 +1\big)^4$ } \\[4pt]
 \hline
  $N$  & 4 &  8 & 16 & 32  & 64 \\
 \hline
  time(s)  &0.0778& 0.3989 & 2.5246 &20.4493  &  166.0483\\
 \hline
 \hline
  \multicolumn{6}{|c|}{  $\beta(\bx)= e^{2(x+y+z)}$ } \\[4pt]
 \hline
  $N$   & 4 & 8 & 16 & 32 &64  \\
 \hline
  time(s) & 0.0770& 0.3982 & 2.2168 &16.3830  &132.7023 \\
 \hline
\end{tabular}
\label{MVTime3D}
\end{table}

\begin{table}[htp]
\centering
\caption{CPU time  for fast multiplication of matrix $B\in \RR^{(N\!-\!1)^3\times(N\!-\!1)^3}$ by any vector.}
\vspace*{2pt}

\begin{tabular}{|p{1.3cm}<{\centering}|p{1.3cm}<{\centering}|p{1.3cm}<{\centering}|p{1.3cm}<{\centering}|p{1.3cm}<{\centering}|p{1.3cm}<{\centering}|}
 \hline
  \multicolumn{6}{|c|}{   $\alpha(x)= \big( 2x^2 + 2y^2 + 2z^2 +1\big)^4$ } \\[4pt]
 \hline
  $N$  & 4 &  8 & 16 & 32  & 64 \\
 \hline
  time(s)  & 0.0251 & 0.1404 & 0.7706 &6.2419  & 50.8091\\
 \hline
 \hline
  \multicolumn{6}{|c|}{  $\alpha(x)= e^{2(x+y+z)}$ } \\[4pt]
 \hline
  $N$   & 4 & 8 & 16 & 32 &64  \\
 \hline
  time(s) & 0.0255& 0.1268 & 0.7127 &5.7871  &47.1070 \\
 \hline
\end{tabular}
\label{MVTime3DB}
\end{table}

The average time of 10 tests of matrix-vector multiplications by matrix $A$ is reported in Table \ref{MVTime1D}, Table \ref{MVTime2D}, Table \ref{MVTime3D}, and by matrix $B$ in Table \ref{MVTime1DB}, Table \ref{MVTime2DB}, Table \ref{MVTime3DB} respectively. It can be observed that the time  scales roughly
linearly in the dimension of  matrices $A$ and $B$, which is consistent with the discussions in section 4.3.

\subsection{ Numerical  results for the number of iterations}

The second test is to demonstrate the effectiveness of  proposed preconditioner $M$.
To this end,  the iteration steps of the PCG method with a constant-coefficient preconditioner (PCG-I) and  the PCG method with the proposed  preconditioner $M$ (PCG-II) are compared. The preconditioner $M$ is constructed by approximating $\beta(\bx)$ and $\alpha(\bx)$ with a ($t_1$+1)-term Legendre series  and a ($t_2$+1)-term Legendre series  in each direction respectively.
In each iteration of PCG-I,
the system with the constant-coefficient preconditioner as the coefficient matrix is solved by direct methods in $\mathcal{O}(N)$ operations for $d=1$ \cite{shen1996} and in $\mathcal{O}(N^d(\log N)^{d-1})$ operations for $d=2,3$\cite{shen1996,Shen1995}.
Numerical results  are presented in Table \ref{Table:Iter1D}, \ref{Table:Iter2D} and \ref{Table:Iter3D}.
Test problems are considered as follows:
\begin{example}
The problem \eqref{NSproblem} in one dimension takes the following coefficients:
\begin{itemize}
 \item[(a)] $\beta(x)=( 2x^2+1)^4$ and $\alpha(x)=\cos(x)$.
 \item[(b)] $\beta(x)=e^{2x}$ and $\alpha(x)=0$.
\end{itemize}
\end{example}

\begin{example}
The coefficients of problem \eqref{NSproblem} in two dimensions are as follows :
\begin{itemize}
\item[(a)] $\beta(\bx)=\big( 2x^2 + 2y^2 +1\big)^4$ and $\alpha(\bx)=\cos(x+y)$.
 \item[(b)] $\beta(\bx)=e^{2(x+y)}$ and $\alpha(\bx)=0$.
\end{itemize}
\end{example}

\begin{example}
The coefficients of problem \eqref{NSproblem} in three dimensions are as follows:
\begin{itemize}
\item[(a)] $\beta(\bx)=\big( 2x^2 + 2y^2 + 2z^2 +1\big)^4$ and $\alpha(\bx)=\cos(x+y+z)$.
 \item[(b)] $\beta(\bx)=e^{2(x+y+z)}$ and $\alpha(\bx)=0$.
\end{itemize}
\end{example}

\begin{table}[htp]
\centering
\caption{Iteration counts for Example\,1.}
\newcommand{\tabincell}[2]{\begin{tabular}{@{}#1@{}}#2\end{tabular}}
\begin{tabular}{|c|c|cccccc|}
  \hline
\multicolumn{8}{|c|}{ Example 1 (a) } \\[4pt]
  \hline
  \multicolumn{2}{|c|}{$N$}
         & $320$ & $640$ & $1280$ & $2560$ & $5120$ & $10240$\\
  \hline
  \multirow{3}{*}[15pt]{PCG-I}
  & $t_1$=0, $t_2$=0 & 130 & 134 & 138 & 142 & 145 & 148\\
    \hline
  \multirow{3}{*}[5pt]{PCG-II}
   & $t_1$=4, $t_2$=2 & 16  & 17   & 17  & 17 &18 & 18\\
    \cline{2-8}
   & $t_1$=6, $t_2$=2 & 7  & 8 & 8 & 8 & 8 &8\\
  \hline
  \hline
  \multicolumn{8}{|c|}{ Example 1 (b) } \\[4pt]
  \hline
  \multicolumn{2}{|c|}{$N$}
         & $320$ & $640$ & $1280$ & $2560$ & $5120$ & $10240$\\
  \hline
  \multirow{3}{*}[15pt]{PCG-I}
  & $t_1$=0, $t_2$=0 & 108& 110 & 113 &116  & 119&121\\
    \hline
  \multirow{3}{*}[5pt]{PCG-II}
   & $t_1$=4, $t_2$=0 & 11  & 11   & 11  & 12 &12 & 12\\
    \cline{2-8}
   & $t_1$=5, $t_2$=0 & 7  & 7 & 7 & 8 & 8 &8\\
  \hline
\end{tabular}
\label{Table:Iter1D}
\end{table}

\begin{table}[htp]
\centering
\caption{Iteration counts for Example\,2.}
\newcommand{\tabincell}[2]{\begin{tabular}{@{}#1@{}}#2\end{tabular}}
\begin{tabular}{|c|c|ccccc|}
  \hline
\multicolumn{7}{|c|}{ Example 2 (a) } \\[4pt]
  \hline
  \multicolumn{2}{|c|}{$N$}
         & 40 & 60 & 80 & 100 & 120 \\
  \hline
  \multirow{3}{*}[15pt]{PCG-I}
  & $t_1$=0, $t_2$=0 & 242 & 281 & 288 & 291 & 292  \\
    \hline
  \multirow{3}{*}[5pt]{PCG-II}
   & $t_1$=4, $t_2$=3 & 15 & 17 & 19 & 21 &23 \\
    \cline{2-7}
   & $t_1$=6, $t_2$=3 & 6 & 7 & 9 & 10 & 10\\
  \hline
  \hline
  \multicolumn{7}{|c|}{ Example 2 (b) } \\[4pt]
  \hline
  \multicolumn{2}{|c|}{$N$}
         & 40 & 60 & 80 & 100 & 120 \\
  \hline
  \multirow{3}{*}[15pt]{PCG-I}
  & $t_1$=0, $t_2$=0 & 545& 585 & 602 &610 &615 \\
    \hline
  \multirow{3}{*}[5pt]{PCG-II}
   & $t_1$=5, $t_2$=0 & 14 & 18 & 22 &26  &30 \\
    \cline{2-7}
   & $t_1$=7, $t_2$=0 & 8 & 11 & 13 &13  &13 \\
  \hline
\end{tabular}
\label{Table:Iter2D}
\end{table}

\begin{table}[htp]
\centering
\caption{Iteration counts for Example\,3.}
\newcommand{\tabincell}[2]{\begin{tabular}{@{}#1@{}}#2\end{tabular}}
\begin{tabular}{|c|c|cccc|}
  \hline
\multicolumn{6}{|c|}{ Example 3 (a) } \\[4pt]
  \hline
  \multicolumn{2}{|c|}{$N$}
         & 12 & 16 & 20 &24  \\
  \hline
  \multirow{3}{*}[15pt]{PCG-I}
  & $t_1$=0, $t_2$=0 & 256 & 366 & 436 & 476   \\
    \hline
  \multirow{3}{*}[5pt]{PCG-II}
   & $t_1$=4, $t_2$=3 & 15 & 19 & 22 & 23  \\
    \cline{2-6}
   & $t_1$=6, $t_2$=3 & 7 & 7 & 8 & 9 \\
  \hline
  \hline
  \multicolumn{6}{|c|}{ Example 3 (b) } \\[4pt]
  \hline
  \multicolumn{2}{|c|}{$N$}
         & 12 & 16 & 20 &24  \\
  \hline
  \multirow{3}{*}[15pt]{PCG-I}
  & $t_1$=0, $t_2$=0 & 1484& 2186 & 2564 & 2744 \\
    \hline
  \multirow{3}{*}[5pt]{PCG-II}
   & $t_1$=5, $t_2$=0 & 7 & 9 & 13 & 16  \\
    \cline{2-6}
   & $t_1$=6, $t_2$=0 & 5 & 8 & 10 & 12  \\
  \hline
\end{tabular}
\label{Table:Iter3D}
\end{table}

Table \ref{Table:Iter1D} reports  the  results for the one-dimensional problem. Note that the PCG method
with the proposed preconditioner $M$ exhibits excellent performance  in terms of iteration step over the PCG method with a constant-coefficient preconditioner. Meanwhile, the iteration steps of  PCG-II   only increase slightly as the discretization parameter $N$ increases.
 As $\beta(\bx)$ and $\alpha(\bx)$ are approximated by a finite number of Legendre series to a higher accuracy, the iteration steps decrease which indicates that the PCG method converges more quickly.
  Table \ref{Table:Iter2D}  and  \ref{Table:Iter3D} list the numerical results for the two- and three-dimensional problems respectively. The iteration steps of PCG-II  for the 2-D and 3-D cases behave similarly as the 1-D case.
All of  the examples show that the proposed  preconditioner $M$ is very effective for problems with large variations in coefficient functions.

\begin{example}
The PCG method with the proposed preconditioner $M$  can be applied to more general second order problems:
\begin{align}
&\begin{cases}\label{problem2}
 -( \beta_1(x)u_x )_x -( \beta_2(y)u_y )_y  + \alpha(\bx)u = f, \quad\bx\in\Omega=(-1,1)^2,\\
 u|_{\partial \Omega}=0,
\end{cases}\\
&\begin{cases}\label{problem3}
 -( \beta_1(x)u_x )_x -( \beta_2(y)u_y )_y -( \beta_3(z)u_z )_z + \alpha(\bx)u = f, \quad\bx\in\Omega=(-1,1)^3,\\
 u|_{\partial \Omega}=0.
\end{cases}
\end{align}

Consider the problem \eqref{problem2} and \eqref{problem3} with the following coefficients:
\begin{itemize}
 \item[(a)] $\beta_1(x)=e^{2x}$,\,$\beta_2(y)=\cos(y)$ and $\alpha(\bx)=0$.
 \item[(b)] $\beta_1(x)=e^{2x}$,\,$\beta_2(y)=\cos(y)$,\,$\beta_3(z)=\cos(z)$ and $\alpha(\bx)=0$.
\end{itemize}
\end{example}

For problems of the form \eqref{NSproblem}, Shen in \cite{shen1996} have pointed out that it is efficient to make a change of dependent variable $v=\sqrt{\beta}u$ \cite{Concus1973} which reduces \eqref{NSproblem} to the following equation:
\begin{align}\label{bianfen}
 \begin{cases}
  -\Delta v +p(\bx)v = q,\quad \bx\in \Omega=[-1,1]^d,\, d=1,2,3,\\
  v|_{\partial \Omega}=0,
 \end{cases}
\end{align}
where $p(\bx)=\frac{\Delta(\sqrt{\beta})}{\sqrt{\beta}}+\frac{\alpha(\bx)}{\beta}$ and $q(\bx)=\frac{f}{\sqrt{\beta}},$
 then the resulting system from the above problem \eqref{bianfen}
  can be solved by using a  preconditioned conjugate gradient  method  with  a constant-coefficient preconditioner. However, this strategy is limited in the situation
such as problem \eqref{problem2} and \eqref{problem3}. In what follows,  both PCG-I and PCG-II are performed  for the linear systems rising from problem \eqref{problem2} and \eqref{problem3}. The preconditioner $M$ is constructed by approximating  $\beta_j$ in case (a) with the $t_{1j}$-term Legendre polynomials, $j$=1,2,   and $\beta_j$ in case (b) with the $t_{1j}$-term Legendre polynomials, $j$=1,2,3.
Numerical results are shown in Table \ref{tabl:Example4a} and Table \ref{tabl:Example4b}.

\begin{table}[htp]
\centering
\caption{Iteration counts for Example\, 4\,(a).}
\newcommand{\tabincell}[2]{\begin{tabular}{@{}#1@{}}#2\end{tabular}}
\begin{tabular}{|c|c|ccccc|}
  \hline
  \multicolumn{2}{|c|}{$N$}
         & 40 & 60 & 80 & 100 & 120 \\
  \hline
  \multirow{3}{*}[15pt]{PCG-I}
  &$t_{11}$=0,\,$t_{12}$=0  & 86 & 89 & 90 & 92 & 92 \\
    \hline
  \multirow{3}{*}[5pt]{PCG-II}
   & $t_{11}$=4,\,$t_{12}$=3  & 10 & 11 &12  & 12 &13\\
    \cline{2-7}
   & $t_{11}$=5,\,$t_{12}$=3  & 8 &10  & 12 & 13 &13 \\
  \hline
\end{tabular}
\label{tabl:Example4a}
\end{table}

\begin{table}[htp]
\centering
\caption{Iteration counts for Example\,4\,(b).}
\newcommand{\tabincell}[2]{\begin{tabular}{@{}#1@{}}#2\end{tabular}}
\begin{tabular}{|c|c|cccc|}
  \hline
  \multicolumn{2}{|c|}{$N$}
         & 12 & 16 & 20 &24  \\
  \hline
  \multirow{3}{*}[15pt]{PCG-I}
  & $t_{11}$=0,\,$t_{12}$=0,\,$t_{13}$=0 &84& 92 &96  & 98 \\
    \hline
  \multirow{3}{*}[5pt]{PCG-II}
   & $t_{11}$=4,\,$t_{12}$=3,\,$t_{13}$=3 & 9 & 10 &10  & 11  \\
    \cline{2-6}
   & $t_{11}$=6,\,$t_{12}$=3,\,$t_{13}$=3 & 6 & 6 & 7 & 8  \\
  \hline
\end{tabular}
\label{tabl:Example4b}
\end{table}

Tables \ref{tabl:Example4a} and \ref{tabl:Example4b} show that  the strategy that using a constant-coefficient problem  precondition variable-coefficient problems is not effective if coefficient functions have large variation over the domain. Further, it indicates  the effectiveness of the  proposed preconditioner $M$.

\subsection{ Numerical ranks of  off-diagonal blocks of matrices in the  Legendre-Galerkin method }
A  direct spectral  method  for differential  equations  with  variable coefficients
in one dimension  was proposed in \cite{SWJ2016}. The strategy therein is based on the rank structures of the matrices in Fourier- and Chebyshev-spectral methods. Numerical ranks of the off-diagonal block $(A)|_{1:\frac{N}{2}, \frac{N}{2}+1:end}$
and  $(B)|_{1:\frac{N}{2}, \frac{N}{2}+1:end}$ with different variable coefficients are computed. Here $A|_{j,k}$ denotes the $(j,k)$ entry of $A$ and can be similarly understood when $j$ and $k$ are replaced by index sets, which is the same as the notation in \cite{SWJ2016}.

\begin{table}[htp]
\centering
\caption{Numerical ranks of the off-diagonal block $(B)|_{1:\frac{N}{2}, \frac{N}{2}+1:end}$ for  $\alpha(x),$ with different sizes $N$ and tolerances $\tau$.}
\newcommand{\tabincell}[2]{\begin{tabular}{@{}#1@{}}#2\end{tabular}}
\scalebox{0.95}{
\begin{tabular}{|l|c|cccccc|}
  \hline
  \multicolumn{2}{|c|}{$N$}
         & $320$ & $640$ & $1280$ & $2560$ & $5120$ & $10240$\\
  \hline
  \multirow{3}{*}[5.5pt]{Numerical rank\big(with $\alpha(x)=\cos(\sin(x))$\big)}
   & $\tau=10^{-6}$ & 6  & 6   & 2  & 2 &2 & 2\\
    \cline{2-8}
   & $\tau=10^{-12}$ & 8  & 8 & 8 & 8 & 8 & 8\\
    \hline
  \multirow{3}{*}[5.5pt]{Numerical rank\big(with $\alpha(x)=e^{x}$\big)}
   & $\tau=10^{-6}$ & 3  & 3   & 3  & 3 &3 & 3 \\
    \cline{2-8}
   & $\tau=10^{-12}$ &5  & 5 & 5 & 5 & 5 &5\\
  \cline{1-8}
  \multirow{3}{*}[5.5pt]{Numerical rank\big(with $\alpha(x)=\frac{1}{100x^2+1}$\big)}
    & $\tau=10^{-6}$ & 2  & 2   & 2  & 2 &2 & 2 \\
    \cline{2-8}
   & $\tau=10^{-12}$ & 2  & 2 & 2 & 2 & 2 & 2\\
  \cline{1-8}
  \hline
\end{tabular}
}
\label{LowRankofmassB}
\end{table}

\begin{table}[htp]
\centering
\caption{Numerical ranks of the off-diagonal block $(A)|_{1:\frac{N}{2}, \frac{N}{2}+1:end}$ for  $\beta(x),$ with different sizes $N$ and tolerances $\tau$.}
\newcommand{\tabincell}[2]{\begin{tabular}{@{}#1@{}}#2\end{tabular}}
\scalebox{0.95}{
\begin{tabular}{|l|c|cccccc|}
  \hline
  \multicolumn{2}{|c|}{$N$}
         & $320$ & $640$ & $1280$ & $2560$ & $5120$ & $10240$\\
  \hline
  \multirow{3}{*}[5.5pt]{Numerical rank\big(with $\beta(x)=\cos(\sin(x))$\big)}
   & $\tau=10^{-6}$ & 6  & 6   & 8  & 8 &8 & 16\\
    \cline{2-8}
   & $\tau=10^{-12}$ & 115  & 291 & 610 & 1260 &2551 & 5116\\
    \hline
  \multirow{3}{*}[5.5pt]{Numerical rank\big(with $\beta(x)=e^{x}$\big)}
   & $\tau=10^{-6}$ & 4  & 4   & 4  & 4 &5 & 20 \\
    \cline{2-8}
   & $\tau=10^{-12}$ & 117  & 292 & 610 & 1260 & 2549 &5116\\
  \cline{1-8}
  \multirow{3}{*}[5.5pt]{Numerical rank\big(with $\beta(x)=\frac{1}{100x^2+1}$\big)}
    & $\tau=10^{-6}$ & 2  & 2   & 2  & 2 &2 & 2 \\
    \cline{2-8}
   & $\tau=10^{-12}$ & 8  & 65 & 374 & 1105 & 2440 &5052\\
  \cline{1-8}
  \hline
\end{tabular}
}
\label{LowRankofstiffA}
\end{table}

Table \ref{LowRankofmassB} indicates  that the off-diagonal numerical ranks of the matrix $B$
do not increase with $N$, indicating the low-rank property for all cases. However,  when $N$ doubles,
the  numerical ranks of $A$ increase  apparently if the accuracy $\tau$ increases from $10^{-6}$ to $10^{-12}$, as is shown in Table \ref{LowRankofstiffA}. Therefore the direct spectral solver based on  the rank structures of the coefficient matrices  is impracticable for the Legendre-Gelerkin method. Besides, the algorithm of constructing an HSS approximation to a dense matrix requires considerable programming effort. Moreover, for two- and three-dimensional problems, a simple HSS structure is generally not practical.


\section{Conclusion}

 An efficient preconditioner $M$ for the  PCG method is proposed for the linear system arising from the Legendre-Galerkin method of second-order elliptic equations.
Since the iteration step of the  PCG method  increase slightly as the discretization parameter $N$ increases,  matrix-vector multiplications can be evaluated in $\mathcal{O}(N^d (\log N)^2)$ operations, and the complexity  of approximately solving the  system with a preconditioner $M$  is of $\mathcal{O}(T^{2d} N^d)$, where the preconditioner $M$ is  constructed by using the ($T$+1)-term Legendre polynomials in each direction to approximate the variable coefficient functions,
the  algorithm admits an $\mathcal{O}(N^d (\log N)^2)$  computational complexity for $d=1,2,3$ while providing spectral accuracy. Furthermore, numerical results indicate that it is very robust.

\section*{Acknowledgements}
We acknowledge the support of  the National Natural
Science Foundation of China (NSFC 11625101 and 11421101).

\end{document}